\documentclass[12pt]{article}   
\usepackage{geometry}

\usepackage{graphicx}				
\usepackage[utf8]{inputenc}
\usepackage[cyr]{aeguill}
\usepackage[english]{babel}	
		
\usepackage{amssymb}
\usepackage{amsmath}
\usepackage{amsthm}
\usepackage{enumerate}
\usepackage{indentfirst}
\usepackage{hyperref}
\usepackage[all]{xy}

\theoremstyle{plain}
\newtheorem{theo}{Theorem}

\newtheorem{prop}[theo]{Proposition}

\newtheorem{lem}[theo]{Lemma}
\theoremstyle{definition}
\newtheorem{defi}[theo]{Definition}

\newtheorem{rem}[theo]{Remark}

\DeclareMathOperator{\id}{Id}

\newcommand{\R}{\mathbf{R}}
\newcommand{\C}{\mathbf{C}}
\newcommand{\Z}{\mathbf{Z}}
\newcommand{\Q}{\mathbf{Q}}

\newcommand{\ora}{\overrightarrow}

\renewcommand{\U}{\mathrm{U}}
\renewcommand{\H}{\mathrm{H}}
\renewcommand{\c}{\ensuremath{\mathcal{C}}}
\renewcommand{\d}{\ensuremath{\mathrm d}}
\renewcommand{\Re}{\mathrm{Re}}
\renewcommand{\phi}{\ensuremath{\varphi}}
\newcommand{\ep}{\ensuremath{\epsilon}}
\newcommand{\la}{\ensuremath{\lambda}}

\author{Alexandre Vérine}
\title{Bohr-Sommerfeld Lagrangian
submanifolds\linebreak as minima of convex
functions}
\begin{document}
\maketitle

\begin{abstract}
We prove that every closed Bohr-Sommerfeld
Lagrangian submanifold $Q$ of a symplectic/Kähler
manifold $X$ can be realised as a Morse-Bott
minimum for some `convex' exhausting function
defined in the complement of a symplectic/complex
hyperplane section $Y$. In the Kähler case,
`convex' means strictly plurisubharmonic while, in
the symplectic case, it refers to the existence of
a Liouville pseudogradient. In particular,
$Q\subset X\setminus Y$ is a regular Lagrangian
submanifold in the sense of
Eliashberg-Ganatra-Lazarev.
\end{abstract}

Rational convexity properties of Lagrangian
submanifolds were first discovered 
in $\C^2$ by Duval, who later investigated them
further with his collaborators
and students. In particular, generalising a result
established by Duval-Sibony \cite{DS95} in
$\C^n$, Guedj \cite{Gue99} obtained the
following theorem: in a complex
projective manifold
$X$, every closed Lagrangian submanifold $Q$ is
rationally convex, which means
that $X\setminus Q$ is filled up with smooth
complex hypersurfaces.
More precisely, these complex hypersurfaces $Y$
are very ample divisors of arbitrarily
large degrees, so their complements are affine
manifolds and possess exhausting
$\C$-convex\footnote
{Here the term \emph{$\C$-convex} is used as a
substitute for ``strictly
plurisubharmonic''.}
functions $f: X \setminus Y \to \R$. In this work,
which was motivated by the
study of vanishing cycles in global
Picard-Lefschetz theory, we give a necessary
and sufficient condition for the existence of such
a function $f$ admitting $Q$
as a Morse-Bott (\emph{i.e.} transversally
non-degenerate) minimum. This
condition refers to a Kähler class and can be more
generally stated as follows
in the symplectic setting:

\begin{defi}
Let $(X,\omega)$ be an integral symplectic
manifold, meaning that $X$ is a
closed manifold and $\omega$ a symplectic form
with integral periods. We say
that a Lagrangian submanifold $Q$ \emph{satisfies
the Bohr-Sommerfeld condition}
--- or simply \emph{is Bohr-Sommerfeld} --- if the
homomorphism $\H_2(X,Q,\Z) \to
\R$ defined by integration of $\omega$ takes its
values in $\Z$. 
\end{defi}

In the Kähler setting, our main result is:

\begin{theo} \label{section hyperplane complexe}
Let $(X,\omega)$ be a closed integral Kähler
manifold and $Q$ a closed
Lagrangian submanifold satisfying the
Bohr-Sommerfeld condition. Then, for every
sufficiently large integer $k$, there exist a
complex hyperplane section $Y$ of
degree $k$ in $X$ avoiding $Q$ and an exhausting
$\C$-convex function $f: X
\setminus Y \to \R$ that has a Morse-Bott minimum
at $Q$ and is Morse away from
$Q$ with finitely many critical points.
\end{theo}
\noindent
To be more explicit, there exists a holomorphic
line bundle $L \to X$ with first Chern class $\omega$ such that the complex
hypersurface $Y$ is the zero-set of a
holomorphic section of the $k$-th tensor power
$L^k$ of $L$.

\medskip

In \cite{AGM01}, Auroux-Gayet-Mohsen reproved
Guedj's above theorem and extended it to
the symplectic setting using the ideas and
techniques developed by Donaldson in
\cite{Don96}. Theorem \ref{section hyperplane
complexe} also has a symplectic
version, whose statement below appeals to the
following terminology:
\begin{itemize}
\item
A \emph{symplectic hyperplane section of degree
$k$} in a closed integral
symplectic manifold $(X,\omega)$ is a symplectic
submanifold $Y$ of codimension
$2$ that is Poincaré dual to $k\omega$.
\item
A function $f: X \setminus Y \to \R$ is
\emph{$\omega$-convex} if it admits a
pseudogradient that is a Liouville (\emph{i.e.} $\omega$-dual to some
primitive of $\omega$) vector field
.
\end{itemize}
With this wording, Donaldson's main theorem in
\cite{Don96} is that every closed
integral symplectic manifold contains symplectic
hyperplane sections of all
sufficiently large degrees. Furthermore, according
to Auroux-Gayet-Mohsen \cite
{AGM01}, such symplectic hyperplane sections can
be constructed away from any
given closed Lagrangian submanifold. On the other
hand, Giroux showed in \cite
{Gir18} that, for all sufficiently large degrees, the
complements of Donaldson's
symplectic hyperplane sections admit exhausting
$\omega$-convex functions (and
hence are Weinstein manifolds). Mixing these
ingredients, we
obtain:

\begin{theo} \label{section hyperplane
symplectique}
Let $(X,\omega)$ be a closed integral symplectic
manifold and $Q$ a closed
Bohr-Sommerfeld Lagragian submanifold of $X$.
Then, for every sufficiently large
integer $k$, there exist a symplectic hyperplane
section $Y$ of degree $k$ in
$X$ avoiding $Q$ and an exhausting $\omega$-convex
function $f: X \setminus Y
\to \R$ that has a Morse-Bott minimum at $Q$ and
is Morse away from $Q$ with
finitely many critical points.
\end{theo}

In \cite{EGL15}, Eliashberg-Ganatra-Lazarev
introduced the following definition: a
Lagrangian submanifold $Q$ in a Weinstein manifold
$(W,\omega)$ is `regular'
if there exists a Liouville pseudogradient on $W$
that is tangent to $Q$ (or equivalently there exists a primitive of $\omega$ vanishing on $Q$). This
property, which implies that $Q$ is an exact
Lagrangian submanifold, is known
for quite a long time to be a strong constraint.
For instance, it is elementary
to see (without any holomorphic curve theory) that
a closed Lagrangian submanifold in
$\C^n$ cannot be regular. In the same time, though
we do not have any example of
a non-regular closed exact Lagrangian submanifold
in a Weinstein manifold, we
do not know any general method to prove that exact
Lagrangian submanifolds should
\emph{a priori} be regular. Theorems \ref{section
hyperplane complexe} and \ref
{section hyperplane symplectique} show that, in
the complement of the complex
and symplectic hyperplane sections constructed,
the Bohr-Sommerfeld Lagrangian
submanifold $Q$ is included in the zero-set of a
Liouville pseudogradient and is
therefore regular.

In section \ref{BSLS} we explain why the
Bohr-Sommerfeld condition is necessary for our
purposes and describe some of properties of
Bohr-Sommerfeld Lagrangians. In section \ref{S} we
prove theorem \ref{section hyperplane
symplectique}, applying the main technical result
from \cite{Gir18}. In section \ref{K} we prove
theorem \ref{section hyperplane complexe} and a
complex-geometric analogue, using techniques that
go back to \cite{DS95}.

\textbf{Acknowledgements.} This work is part of my
Ph.D. prepared at ÉNS de Lyon under the
supervision of Emmanuel Giroux. I warmly thank
him for his help and support and Jean-Paul Mohsen for his comments
on a draft of this paper. This work was
supported by the LABEX MILYON (ANR-10-LABX-0070)
of Université de Lyon, within the program
``Investissements d'Avenir" (ANR-11-IDEX-0007)
operated by the French National Research Agency
(ANR), and by the UMI 3457 of CNRS-CRM.

\section{Bohr-Sommerfeld Lagrangian
submanifolds...} \label{BSLS}

Let us first remark that Cieliebak-Mohnke proved,
in \cite[Theorem 8.3]{CM17}, a version of the main
theorem of \cite{AGM01} that is specific to
Bohr-Sommerfeld Lagrangian submanifolds.

The Bohr-Sommerfeld condition in theorems
\ref{section hyperplane complexe} and 
\ref{section hyperplane symplectique} is
necessary, indeed:
\begin{lem}
Let $(X,\omega)$ be a closed symplectic manifold
and $Q$ a Lagrangian submanifold. Suppose that
there exist a symplectic hyperplane section
$Y\subset X$ of degree $k$ avoiding $Q$ and $\la$
a primitive of $\omega$ over $X\setminus Y$ such
that $\la|_Q$ is exact. Then $Q$ is a
Bohr-Sommerfeld Lagrangian submanifold of
$(X,k\omega)$.
\end{lem}

\begin{proof}

It suffices to prove the following (well-known)
claim: Let $X$ be a closed connected oriented
manifold, $Y\subset X$ a closed codimension $2$
submanifold and $\omega$ a non-exact closed
$2$-form on $X$ that is Poincaré-dual to $Y$.
Then, for every compact surface $\Sigma\subset X$ with boundary disjoint from $Y$ and primitive
$\la$ of $\omega$ on $X\setminus Y$ such that $\la|_Q$ is
exact,
$$\int_\Sigma \omega=\Sigma.Y \ .$$

We first suppose that $Y$ is connected. For any
embedded 2-disc $D$ intersecting $Y$ transversely
at one point, with sign $\ep(D)=\pm 1$, set
$r:=\ep(D)( \int_D \omega -\int_{\partial D}
\la)$. The `residue' $r$ does not depend on the
disc $D$. Indeed, for two such discs $D$ and $D'$,
connectedness of $Y$ gives an oriented cylinder
$C$ in $X\setminus Y$ bounding $-\ep(D')\partial
D'$ and $\ep(D) \partial D$ so by Stokes theorem,
$$\ep(D')\int_{\partial D'} \la -\ep(D)
\int_{\partial D} \la  = \int_C \omega =
\ep(D')\int_{D'} \omega -\ep(D) \int_{ D} \omega.$$

Let $\Sigma\subset X$ be a compact surface
intersecting $Y$ away from $\partial \Sigma$. By a
general position argument we may suppose the
intersection is transverse. For each point $p_i\in
\Sigma \cap Y$, take a disc $D_i\subset \Sigma$
that intersects $Y$ only at $p_i$. Stokes theorem
gives $\int_{\Sigma\setminus \cup_i D_i}\omega = -
\sum_i \int_{\partial D_i} \la$, then:
\begin{align}\label{Stokes}
\int_{\Sigma}\omega = \Sigma.Y \ r \ . 
\end{align}
Since $\omega$ is not exact, we can apply
(\ref{Stokes}) to some closed surface $\Sigma_0$
with $\Sigma_0.Y=\int_{\Sigma_0}\omega\linebreak \neq 0$.
This gives $r=1$; so (\ref{Stokes})
proves the claim.

Suppose $Y$ is not connected. If $\dim X\geq 3$,
the cycle $[Y]$ may be represented by a closed
connected submanifold, namely an embedded (away
from $\partial \Sigma_0$) connected sum of the
connected components of $Y$. If $\dim X=2$, we may
represent $[Y]$ by some integral multiple of any
point. Consequently, we reduce to the previous
case.

\end{proof}

Meanwhile, the Bohr-Sommerfeld condition can be
easily obtained after a modification of the
symplectic form:
\begin{lem}[Approximation and rescaling]
 Let $(X,\omega)$ be a closed symplectic manifold
and $Q$ a closed Lagrangian submanifold. Then
there exists a small closed $2$-form $\ep$ and an
integer $k$ such that $Q$ is a Bohr-Sommerfeld
Lagrangian submanifold of $(X,k(\omega+\ep))$.
\end{lem}
\begin{proof}
We argue as in \cite{AGM01}: the $2$-form $\omega$
vanishes on $Q$ so, in view of the exact sequence
$\dots\to H^2(X,Q;\R)\to H^2(X;\R)\to
H^2(Q;\R)\to\cdots $, it is the image of a
relative class $c\in H^2(X,Q;\R)$. We approximate
$c$ by some $r\in H^2(X,Q;\Q)$ and take a small
closed form $\ep$ vanishing on $Q$ that represents
$c-r$. Then the closed form $\omega-\ep$ is
symplectic, vanishes on $Q$ and its relative
periods --- given by evaluation of $r$ --- are
rational.  
\end{proof}

We now give the characterisation of
Bohr-Sommerfeld Lagrangian submanifolds that we
will use to prove theorems \ref{section hyperplane
complexe} and \ref{section hyperplane
symplectique}.
\begin{lem}[Hermitian flat line
bundles]\label{fibré}
Let $(X,\omega)$ be an integral symplectic manifold and $Q$
a submanifold. Then $Q$ is a Bohr-Sommerfeld
Lagrangian submanifold if and only if there exist
a Hermitian line bundle $L \to X$ and a unitary
connection $\nabla$ of curvature $-2i\pi \omega$
such that $(L,\nabla)|_Q$ is a trivial flat
bundle. If $Q$ is a Bohr-Sommerfeld Lagrangian
and, in addition, $(X,\omega)$ is Kähler, then one
can take for $(L,\nabla)$ a holomorphic Hermitian
line bundle with its Chern connection.
\end{lem}

\begin{proof}

Suppose that $Q$ is a Bohr-Sommerfeld Lagrangian
submanifold. Since $\omega$ has integral periods,
we may fix a lift $c$ of its cohomology class to
$\H^2(X, \Z)$. We take a Hermitian line bundle
$L_0\to X$ with first Chern class $c$ and a
unitary connection $\nabla_0$ of curvature 
$-2i\pi \omega$. The submanifold $Q$ is Lagrangian
so the restriction $(L_0,\nabla_0)|_Q$ is a flat
Hermitian bundle.

We will construct a flat Hermitian line bundle
$(L_1,\nabla_1)\to X$ whose restriction to $Q$ is
isomorphic to $(L_0,\nabla_0)|_Q$. Then the
desired line bundle will be $L_0\otimes L_1^{-1}$.

Recall that flat Hermitian line bundles over a
manifold $Y$ are classified up to isomorphism by
their holonomy representation $\H_1(Y,\Z)\to
\U(1)$ (cf. proposition 3.6.15 in \cite{Thu97}).
To construct the flat bundle $(L_1,\nabla_1)$ it
suffices to extend the holonomy representation
$\rho: \H_1(Q,\Z)\to \U(1)$ of the flat bundle
$(L_0,\nabla_0)|_Q$ to a homomorphism
$\H_1(X,\Z)\to \U(1)$.

We first show that $\rho$ vanishes on the kernel
of the group homomorphism $i: \H_1(Q,\Z)\to
\H_1(X,\Z)$ induced by inclusion. Consider the
exact sequence of the pair $(X,Q)$: 
  \begin{align*}
\cdots \rightarrow \H_2(X,Q; \Z) \xrightarrow
\partial \H_1(Q,\Z) \xrightarrow i \H_1(X,\Z)
\rightarrow \cdots
  \end{align*}
where $\partial$ is the homomorphism given by the
boundary of chains. It suffices to show that
$\rho\circ \partial =0$. Every $a\in \H_2(X,Q;\Z)$
can be represented by an embedded surface
$\Sigma\subset X$ whose (possibly empty) boundary
is included in $Q$. It then follows from (well-known) lemma
\ref{Gauss Bonnet} that:
\begin{align}\label{holonomie:courbure}
\rho(\partial a) = \exp\left(2i\pi\int_a
\omega\right) .
\end{align}
Since the Lagrangian submanifold $Q$ is
Bohr-Sommerfeld, $\rho(\partial a)=0$.

Thus $\rho$ factors through a homomorphism 
$\tilde \rho : \H_1(Q,\Z)/\ker i \to \U(1)$ where
$\H_1(Q,\Z)/\ker i$ injects into $\H_1(X,\Z)$. Now
$\U(1)$ is a divisible abelian group so it is an
injective $\Z$-module (see for instance
\cite[Corollary 2.3.2]{Wei95}). Hence $\tilde
\rho$ extends to $H_1(X,\Z)$.

In the case where $(X,\omega)$ is Kähler, the above
Hermitian line bundle $(L_0,\nabla_0)$ can be chosen
holomorphic with its Chern connection (see, \emph{e.g.}, \cite[Theorem
13.9.b]{Dem12}). On the other hand the flat line
bundle $(L_1,\nabla_1)$ is isomorphic to the
quotient of the trivial flat bundle $\tilde
X\times \C$ by the diagonal action of the
fundamental group, acting on its universal cover $\tilde X$ by deck
transformations and on $\C$ by the holonomy
representation $\H_1(X,\Z)\to \U(1)$ (cf.
proposition 3.6.15 in \cite{Thu97}). Therefore the
trivial holomorphic structure and the trivial
connection on $\tilde X\times \C$ respectively
induce a holomorphic structure and the Chern
connection on $L_1$. Consequently, the bundle $L_0\otimes
L_1^{-1}$ has the desired properties.

Conversely, let $(X,\omega)$ be a symplectic
manifold and a Hermitian line bundle $L\to X$ with
a unitary connexion of curvature $-2 i\pi \omega$
such that $(L,\nabla)|_Q$ is a trivial flat
bundle. Then the (trivial) holonomy representation
$\rho$ of $(L,\nabla)|_Q$ satisfies
(\ref{holonomie:courbure}); so $Q$ is a
Bohr-Sommerfeld Lagrangian.
\end{proof}

\begin{lem}[Gauss-Bonnet] \label{Gauss Bonnet} 
Let $X$ be a manifold and $L\to X$ a Hermitian
line bundle with a unitary connection $\nabla$
whose curvature $2$-form is written $-2i\pi
\omega$. Let $\Sigma$ be a connected oriented
surface with non-empty boundary and $f:\Sigma \to
X$ a map. The holonomy of $\nabla$ along the loop
$f|_{\partial \Sigma}$ is $\exp(2i\pi
\int_{\Sigma}f^\star \omega)\in \U(1)$.
\end{lem}

\begin{proof}
We may assume $X=\Sigma$ and $f=\id_\Sigma$ by
pulling back the line bundle $L$ by $f$. There is
a unit section $s:\Sigma\to L$. In the
trivialisation of $L$ given by $s$ there is a
primitive $\alpha$ of $\omega$ such that the
connection $\nabla$ reads $d -2i\pi \alpha$. By
Stokes theorem
\begin{align*}
 \int_{\Sigma} \omega = \int_{ \partial \Sigma}
\alpha  \  .
\end{align*}
We may assume that $\partial \Sigma$ is connected.
Take $\beta: [0,1] \to \partial \Sigma$ a
parametrisation of $\partial \Sigma$. For every
unit parallel lift $\gamma: [0,1]\to L$ of
$\beta$ and for all $t\in [0,1]$, $\gamma'(t)=
2i\pi \gamma(t) \ (\beta^\star \alpha)_t(\partial_t)
$ hence
$$ 2i\pi  \int_{ \partial \Sigma} \alpha  = 
\int_{[0,1]} \frac{\gamma'(t)}{\gamma(t)} \ dt =
\log \frac{\gamma(1)}{\gamma(0)} \ .$$
An exponentiation gives the result.
\end{proof}

\section{...As minima of \texorpdfstring{$\omega$}{omega}-convex
functions}\label{S}
In this section we prove theorem \ref{section
hyperplane symplectique} so $Q$ is a closed
Bohr-Sommerfeld Lagrangian submanifold in a closed
integral symplectic manifold $(X,\omega)$. We fix
an $\omega$-compatible almost complex structure
$J$ on $X$ and denote by $g=\omega(\cdot,J\cdot)$
the corresponding Riemannian metric. We denote by
$\la_0$ the Liouville form on $T^\star Q$. The
metric induced by $g$ on $Q$ determines a norm
$|\cdot|$ on each fibre of $T^\star Q\to Q$. We
define the function $f_0 : T^\star Q\to \R_{\geq
0}$ by $f_0(p):=\pi |p|^2$. Using Weinstein's
normal form theorem, we identify a
neighbourhood $N$ of $Q\subset (X,\omega)$ with a
tube $\{f_0<c\}$ around the zero section $Q$
in $(T^\star Q,\d\la_0 )$ in such a way that, for all $q\in
Q$, the subspaces $T_qQ,T_q^\star Q\subset T_q(T^\star Q)$ are $g$-orthogonal. 

Using Lemma \ref{fibré}, we fix a Hermitian line
bundle $L\to X$ with a unitary connection $\nabla$
of curvature $-2\pi i \omega$ and a unit parallel
section $s_0$ of the flat bundle $(L,\nabla)|_Q$.
Parallel transport by $\nabla$ along the rays in
the fibres of $T^\star Q$ extends $s_0$ to a
section $s:N\to L|_N$. We extend the section
$s_0$ by $s_0:=e^{-f_0}s:N\to L|_N$.

In the following, we denote by $L^k$ the $k$-th
tensor power of the line bundle $L$, the induced
connection has
curvature $-2k\pi i \omega$. We set
$g_k:=kg$ the rescaled metric. For any integer $r\ge 0$, we endow the
vector bundle $\bigotimes^r T^\star X \otimes L^k$
with the connection induced by the Levi-Civita
connection for the metric $g_k$ and our connection
on $L^k$; we still write this connection $\nabla$.
The $J$-linear and
$-J$-linear parts of the connexion $\nabla$ are
written $\nabla'$ and $\nabla''$.
We define the $\c^r$ norm of a section $ u:X\to
L^k$ by $\Vert u\Vert_{\c^r,g_k}:=\sup |u|+\sup
|\nabla u|_{g_k}+ \dots+\sup |\nabla^r u|_{g_k}$.

For any $1$-form $\la$ on $X$, we will denote by
$\ora\la$ the vector field that is $k\omega$-dual
to $\la$, where $k$ will be given by the context.
\begin{lem}\label{quasihol}
There exists a constant $C>0$ such that, for every
integer $k\ge 1$, the function $f_0$ and the
section $s_0^k$ satisfy the following bounds on
$N$: 
\begin{align*}
\ora{\la_0} .(kf_0) & \ge
C^{-1}(|\ora{\la_0}|_{g_k}^2 + |\d
(kf_0)|_{g_k}^2), \ C^{-1} (kf_0)^{1/2} \le |\d
(kf_0)|_{g_k} \le C (kf_0)^{1/2},\\
 |\nabla s_0^k|_{g_k}&\le C(kf_0)^{1/2}e^{-kf_0},\
\|\nabla^2 s_0^k\|_{\c^0,g_k}\le C \text{ and }
\|\nabla'' s_0^k\|_{\c^1,g_k}\le C k^{-1/2}.
\end{align*}

\end{lem}

\begin{proof}
 
By rescaling, it suffices to establish the first two bounds of the statement
for $k=1$. The function $f_0$ is Lyapounov for the vector
field $\ora{\la_0}$. This implies the first bound.
The submanifold $Q$ is a Morse-Bott minimum for
$f_0$, hence the second bound.

Since $s_0= e^{-f_0}s$ with $s$ parallel,
$$\nabla s_0=  -\d f_0 e^{-f_0}s+ e^{-f_0}\nabla s =-(\d
f_0+2\pi i\la_0) s_0.$$
Therefore, $\nabla s_0$ vanishes identically on the
zero section. Hence, there exists a constant $C>0$
such that $|\nabla s_0|_g\le C f_0^{1/2}$.
Moreover, the $1$-jet of $\nabla'' s_0$ vanishes
identically on $Q$. Indeed, by the identity $\la_0
= -\omega(\cdot,\ora{\la_0})$ (here $k=1$) and by $J$-linearity
of the $1$-form
$g(\cdot,\ora{\la_0})-i\omega(\cdot,\ora{\la_0})$,
$$\nabla''s_0 = -2\pi( \tfrac{\d f_0}{2\pi}
+i\la_0))'' s_0 = -2\pi( \tfrac{\d
f_0}{2\pi}-g(\cdot,\ora{\la_0}))'' s_0,$$
so it suffices to show that the $1$-jet of the
$1$-form $\tfrac{\d
f_0}{2\pi}-g(\cdot,\ora{\la_0})$ vanishes
identically along $Q$. Its $0$-jet clearly
vanishes, and, for each vector $v=(v_1,v_2)$ in the $g$-orthogonal sum $T(T^\star
Q)|_Q= TQ\oplus T^\star Q$,
$$\d (g(\cdot,\ora{\la_0}))(v,v) = g(v,v.
\ora{\la_0}) = g(v,v_2)=g(v_2,v_2)=(\d^2
f_0)(v,v)/(2\pi),$$ 
hence its $1$-jet vanishes too.
Consequently, there exists a constant $C>0$ such
that $|\nabla \nabla'' s_0|_g\le C f_0^{1/2}$ and
$|\nabla'' s_0|_g\le C f_0$. Therefore, by the
Leibniz rule, we obtain the desired bounds on
$\nabla s_0^k$ and $\nabla^2 s_0^k$, and the two
bounds $|\nabla'' s_0^k|_{g_k}\le Ck^{1/2} f_0
e^{-kf_0}$, $|\nabla \nabla'' s_0^k|_{g_k}\le (k
f_0^{3/2} + f_0^{1/2}) Ce^{-kf_0}$. The two latter
real-valued Gaussian functions of $f_0$ both reach
their global maximum at $Constant \times k^{-1}$
so we obtain the last bound of the statement.

\end{proof}

We will state next lemma using Auroux's following 
\begin{defi}
Sections $s_k:X\to L^k$ are called
\emph{asymptotically holomorphic} if there exists
a constant $C>0$ such that for every positive
integer $k$, $\Vert \nabla''
s_k\Vert_{\c^1,g_k}\le C k^{-1/2}$ and $\Vert
s_k\Vert_{\c^2,g_k}\le C$.
\end{defi}

The following result was already observed in Auroux-Gayet-Mohsen \cite[Remark p.746]{AGM01}. Recall that our neighbourhood $N$ of $Q$ is
identified with the cotangent tube $\{f_0<c\}$.
\begin{lem}\label{remAGM}
Let $\beta: N\to [0,1]$ be a
compactly supported function (independent of $k$)
with $\beta=1$ on a tube $\{f_0<b\}$. Then, the sections
$s_{0,k}:= \beta s_0^k:X\to L^k$ are
asymptotically holomorphic.
\end{lem}

\begin{proof}[Proof of lemma \ref{remAGM}]

The sections $s_0^k$ satisfy the estimates of lemma \ref{quasihol} on
$N$. Then, there exists a constant $C>0$ such
that:
\begin{align*}
\|\nabla'' s_{0,k}\|_{\c^0,g_k}
&\le  \|d\beta\|_{\c^0,g_k} \sup_{\{f_0>b\}}
|s_0^k| + \|\nabla''s_0^k\|_{\c^0,g_k} \\
&\le C k^{1/2}e^{-bk} + Ck^{-1/2}.
\end{align*}
Similarly:
\begin{align*}
 \|\nabla\nabla'' s_{0,k}\|_{\c^0,g_k}
 &\le \|d^2\beta\|_{\c^0,g_k} \sup_{\{f_0>b\}}
|s_0^k| + 2\|d\beta\|_{\c^0,g_k} \sup_{\{f_0>b\}}
|\nabla s_0^k|_{g_k}+
\|\nabla\nabla''s_0^k\|_{\c^0,g_k} \\
 &\le Cke^{-bk} + 2 C
k^{1/2}(Ck^{1/2}c^{1/2}e^{-bk})+ Ck^{-1/2}.
\end{align*}
Hence, there exists a constant $C>0$ such
that, for all $k$,
$\|\nabla''s_{0,k}\|_{\c^1,g_k}\le Ck^{-1/2}$. In
the same way, we obtain the bound
$\|s_{0,k}\|_{\c^2,g_k}\le C$.

\end{proof}

Giroux's theorem below provides transverse perturbations our sections $s_{0,k}$
with the following property
\begin{defi}[Giroux]
Let $\kappa\in (0,1)$. A section $s:X\to L^k$ is
called $\kappa$-quasiholomorphic if $|\nabla''
s|\le \kappa |\nabla' s|$ at each point of $X$.
\end{defi}

\begin{theo}[{\cite[Proposition 13]{Gir18}}]
\label{perturbation}
Let $\ep>0$, $\kappa\in (0,1)$ and $s_{0,k}:X\to
L^k$ asymptotically holomorphic sections. Then,
for any sufficiently large integer $k$, there
exists a section $s_{1,k}:X\to L^k$ with the
following properties: 
\begin{itemize}
\item
$s_{1,k}$ vanishes transversally;
\label{transversalité}
\item
$s_{1,k}$ is $\kappa$-quasiholomorphic;
\label{kappa}
\item $\Vert s_{1,k}-s_{0,k}\Vert_{\c^1,g_k}<\ep$
; \label{s_1-s_0}
\item $-\log |s_{1,k}| : \{p\in X, \ s(p)\neq
0\}\to \R$ is a Morse function. \label{Morse}
\end{itemize}
\end{theo}

Let us now bring the arguments together to prove theorem \ref{section
hyperplane symplectique}.
\begin{proof}[Proof of theorem \ref{section
hyperplane symplectique}]
 
Using lemma \ref{remAGM}, we fix sections $s_{0,k}:X\to L^k$ with $s_{0,k}=s_0^k$ on a tube $\{f_0<b\}$. We then fix $\ep\in (0,1)$ and take
sections $s_{1,k}:X\to L^k$ provided by theorem
\ref{perturbation}.
The subset $Y:=\{s_{1,k}=0\}\subset (X,\omega)$ is
a symplectic hyperplane section of degree $k$
(because of the first two properties of theorem
\ref{perturbation}, see for instance proposition
$3$ in \cite{Don96}) avoiding the submanifold $Q$
(because $|s_0|=1$ on $Q$ and by the third
property of theorem \ref{perturbation}).

It remains to construct an $\omega$-convex
exhaustion $f:X\setminus Y\to \R$ that has a
Morse-Bott minimum at $Q$ and is Morse away from
$Q$ with finitely many critical points. On the one hand, the function
$f_{0,k}:=kf_0:N\to \R$ has a Morse-Bott minimum at $Q$
 and, by Lemma \ref{quasihol}, is Lyapounov for the Liouville vector field $\ora{\la_{0}}$
with a uniform Lyapounov constant in the metric $g_k$. On the other hand, the exhaustive function $f_{1,k}:=-\log|s_{1,k}|$ is
Morse (by the last property of theorem
\ref{perturbation}) and has finitely many critical
points (because $s_{1,k}$ vanishes transversally).
A pseudgradient for $f_{1,k}$ is provided by Giroux's following lemma. Before stating it, we set $\la_{1,k}$ the
real $1$-form such that, in the unitary
trivialisation of $L^k|_{X\setminus Y}$ given by
$s_{1,k}/|s_{1,k}|$, the connection $\nabla$ reads
$\d-2k\pi i \la_{1,k}$. We also recall that the
notation $\ora\la$ stands for the $k\omega$-dual
vector field to a given $1$-form $\la$, where $k$
is given by the context.
\begin{lem}[{\cite[Lemma
12]{Gir18}}]\label{pseudogradient}
Let $\kappa\in(0,1)$ and $s_{1,k}:X\to L^k$ a
$\kappa$-\linebreak quasiholomorphic section. Then
 $$\ora{\la_{1,k}}. f_{1,k} \geq
\tfrac{1}{2}\frac{1-\kappa^2}{1+\kappa^2} (|\d
f_{1,k}|_{g_k}^2+|\ora{\la_{1,k}}|_{g_k}^2).$$
\end{lem}
Hence the function $f_{1,k}$ is Lyapounov for the
Liouville vector field
$\ora{\la_{1,k}}$, with a uniform Lyapounov
constant in the metric $g_k$. Finally, the desired function $f$ is constructed in the following lemma, by
gluing, on an annular region $\{a<f_{0,k}<b\}$
about $Q$, the standard (Morse-Bott) Weinstein
structure $(\ora{\la_{0}},f_{0,k})$ on $T^\star Q$
with the Weinstein structure
$(\ora{\la_{1,k}},f_{1,k})$ given by Giroux's
above theorem.
\end{proof}
\begin{lem}\label{glue:W}
Let $\kappa\in(0,1)$ and $a,b\in(0,c)$ with $a<b$.
Then, for every sufficiently small $\ep\in (0,1)$
and for every $k\geq k_0(\ep)$ sufficiently large,
there exist a Liouville vector field $\ora\la$ on
$X\setminus Y$ and a Lyapounov function
$f:X\setminus Y\to \R$ for $\ora\la$ such that
$(\ora{\la},f)=(\ora{\la_{0}},f_{0,k})$ on
$\{f_{0,k} \le a\}$ and
$(\ora\la,f)=(\ora{\la_{1,k}},f_{1,k})$ away from
$\{f_{1,k}<b\}$.
\end{lem}

\begin{proof}
We will omit the indices $k$ in the proof. We
observe first that, for sufficiently small $\ep$,
the section $s_1$ does not vanish on $N$ and the
tube $\{ f_1<b \}$ is contained in $N$ (because
$\|s_1-s_0\|_{\c^0}<\ep$ and $\inf_{\{f_0< c \}}
|s_{0}|>e^{-c}$).

Let us prove that there exists a constant $C>0$
(independent of $k,\ep$) such that
\begin{align}\label{fla}
\|f_0-f_1\|_{\c^1(N),g_k}\le C \ep
\end{align}
and that, for sufficiently small $\ep>0$, the form
$\la_1-\la_0$ is exact on $N$.
On the one hand, $f_0-f_1 = \mathrm{Re}\log (s_1
s_0^{-1})$ and, since $u_j:=s_j/|s_j|$ satisfies
$\nabla u_j= -2k\pi i \la_j u_j$,
$$\la_1-\la_0 
=\tfrac{1}{2k\pi i}\left(u_0^{-1}\nabla
u_0-u_1^{-1}\nabla u_1\right) 
=\tfrac{1}{2k\pi i}\d\log (u_0 u_1^{-1})
=\tfrac{1}{2k\pi} \d\arg (s_1 s_0^{-1}) .$$
On the other hand, $\|\log (s_1
s_0^{-1})\|_{\c^1,g_k}\le C\ep$; this is a
consequence of the three bounds
$\|s_1-s_0\|_{\c^1,g_k}<\ep$, $\inf
|s_{0}|>e^{-c}$, and $\|\nabla s_0\|_{\c^0,g_k}
\le Constant$ (from lemma \ref{quasihol}). In
particular we obtain the bound (\ref{fla}) and,
for $\ep$ sufficiently small, $\|\arg (s_1
s_0^{-1})\|_{\c^0}<\pi/3$ so $\la_1-\la_0$ is
exact.

Now, for $\ep<(b-a)/2$, the annular region $\{ a <
f_0 < b \}$ contains the level set $\{f_1=
(b-a)/2\}$ (by the bound (\ref{fla})). We
construct a Lyapounov function $f:X\to \R$ for the
vector field $\ora{\la_0}$ with $f=f_0$ on
$\{f_0\le a\}$ and $f=f_1$ away from $\{f_1 < 
(b-a)/2\}$. It suffices to show that $\ora{\la_0}$
is transverse to the level sets of $f_0$ and $f_1$
in the region $\{ a < f_0 , f_1< (b-a)/2 \}$;
indeed, increasing from $a$ to $b$ along each
trajectory of $\ora{\la_0}$ gives a function $f$
transverse to the level sets of $f_0$ and $f_1$.
There exists a constant $C'>0$ such that $\|\d
f_0\|_{g_k}\ge C'$ (by lemma \ref{quasihol}). By
the latter bound and (\ref{fla}), $\|\d
f_1\|_{g_k}\ge C'-C\ep$. So, by the Lyapounov
conditions, there exists a constant $C'>0$ such
that $\ora{\la_0}.f_0\ge C'$ and
$\ora{\la_1}.f_1\ge C'$. By the latter bound and
again (\ref{fla}), $\ora{\la_0}$ is transverse to
the level sets of $f_1$.

It remains to construct a Liouville vector field
$\ora\la$ transverse to the level sets
of the function $f$ in the annular region
$\{(b-a)/2 <f_1<b \}$ (where $f=f_1$), and coïnciding with $\ora{\la_0}$ on
$\{f_1<(b-a)/2\}$ and with $\ora{\la_1}$ outside
$\{f_1\le b\}$. In view of
our initial computation, we have a function $H$
such that
$\ora{\la_1}-\ora{\la_0}=\overrightarrow{\d H}$.
Let us fix a cutoff function $\beta:\R\to [0,1]$
such that $\beta=0$ near $\R_{\le (b-a)/2}$ and
$\beta=1$ near $\R_{\ge b}$ and set
$\beta_1:=\beta\circ f_1$. Then the vector field
$\ora\la := \ora{\la_0} +
\overrightarrow{\d(\beta_1 H)}$ is Liouville and
satisfies the desired boundary conditions.
Moreover, $\overrightarrow{\beta_1}$ is tangent to
the level sets of $f_1$ and, by the above
paragraph, $\ora{\la_0}$ and $\ora{\la_1}$ are
transverse to these, so
$\ora\la=(1-\beta_1)\ora{\la_0} +
\beta_1\ora{\la_1} +\overrightarrow{\beta_1} H$ is
transverse to these too.

\end{proof}

\begin{rem}[An alternative proof of the regularity of $Q\subset X\setminus Y$]
For sufficiently large $k$, it is possible to
choose our $\kappa$-quasiholomorphic perturbation
$s_{1,k}:X\to L^k$ (vanishing transversally and
away from $Q$) of $s_{0,k}$ in such a way that the
quotient function $(s_{1,k}/s_{0,k})|_Q$ is
real-valued. The latter property, which can be
achieved by implementing techniques from
Auroux-Munoz-Presas' \cite{AMP05} in the proof of
\cite[Proposition 13]{Gir18}, implies that the
Liouville pseudogradient $\ora{\la_{1,k}}$ of the
function $-\log |s_{1,k}|$ is tangent to $Q$.
\end{rem}

\section{...As minima of \texorpdfstring{$\C$}{C}-convex
functions}\label{K}

This section deals with the proof of Theorem
\ref{section hyperplane complexe},
so $Q$ is a closed Bohr-\linebreak Sommerfeld
Lagrangian submanifold in a closed integral
Kähler manifold $(X,\omega)$. Using Lemma
\ref{fibré}, we fix a holomorphic
Hermitian line bundle $L \to X$ with Chern
curvature $-2\pi i \omega$ and a
parallel unit section $s_0:Q\to L|_Q$. We denote by $\nabla$ the
Chern connection. We denote
by $d_k$ the distance function
to $Q$ in the metric $g_k =
k\omega(\cdot,i\cdot)$ and we set $B_k(Q,c):=\{d_k< c\}$.  We recall that we endow the
vector bundle $\bigotimes^r T^\star X \otimes L^k$
with the connection induced by the Levi-Civita
connection for the metric $g_k$ and the connection
on $L^k$ --- we still write this connection
$\nabla$. We define the $\c^r$ norm of a section $
u:X\to L^k$ by $\Vert u\Vert_{\c^r,g_k}:=\sup
|u|+\sum_{j=1}^r\sup |\nabla^j u|_{g_k}$.

Since $Q$ is a totally real submanifold of $X$, it
has a neighbourhood on which the squared distance
function $d_1^2$ is $\C$-convex (see for instance
Proposition $2.15$ in \cite{CE12}), so some neighbourhood $N:=B_1(Q,c)$
 is a Stein
manifold. By results of Oka \cite
{Oka39} and Grauert \cite{Gra58}, the line bundle
$L|_N$ admits a non-vanishing
holomorphic section $s: N \to L|_N$. Furthermore,
given any integer $r\ge 1$, \cite[Proposition
5.55]
{CE12} shows that the complex-valued function $s_0
/ (s|_Q)$ extends to a smooth function
$F: N \to \C$ such that the form $\d''F$ vanishes
identically
along $Q$ together with its $r$-jet. We will
eventually choose $r=n$, the complex dimension of
the manifold $X$. We then extend $s_0$ over $N$ by
setting $s_0 := Fs: N \to L|_N$.

\begin{rem}[The real-analytic case]
If the submanifold $Q$ is
real-analytic, then one can take for $s_0: N \to L|_N$
a holomorphic section. Indeed, one may ensure that the connection
$\nabla$ on the bundle $L$ provided by Lemma
\ref{fibré} is real-analytic. In that case, the section
$s_0:Q\to L|_Q$ is real-analytic and can be complexified.
\end{rem}

\begin{lem} \label{decroissance:exponentielle}
There exist a constant $C>0$
such that, for every integer $k \ge 1$, the
section $s_0^k: N \to L^k|_N$ satisfies the
following bounds on $N$:
\begin{align*}
    |2\pi k\omega-\d\d^c \log|s_0^k||_{g_k}&\le C
k^{-1/2} d_k, \\
    e^{-Cd_k^2}  \le |s_0^k| &\le  e^{-d_k^2/C},
\\
    |\nabla s_0^k|_{g_k} \le  C d_k e^{-d_k^2/C}, 
   \|\nabla''s_0^k\|_{\c^1,g_k} &\le C  k^{-r/2}, 
k^{-r/2}. 
\end{align*}
\end{lem}

\begin{proof}
We first observe that $\nabla s_0$ vanishes at
every point $p \in Q$. Indeed,
$T_pX \linebreak = T_pQ \oplus i\, T_pQ$ (because $Q$ is
totally real of middle dimension),
$\nabla s_0(p) = \nabla's_0(p)$ (because
$\nabla''s_0(p)=0$) and $\nabla s_0(p)$
vanishes on $T_pQ$ (because $s_0|_Q$ is parallel).
Thus, there exists a constant
$C>0$ such that $|\nabla s_0| \le Cd_1$.
Similarly, since the $r$-jet of
$\nabla''s_0$ vanishes identically on $Q$, there
exists a constant $C>0$ such that
$|\nabla''s_0|_{g_1} \le Cd_1^{r+1}$ and
$|\nabla\nabla''s_0|_{g_1} \le
Cd_1^{r}$.

We now consider the function $f_0 := -\log|s_0|$.
Clearly, $f_0(p) = 0$ and 
$$ \d f_0(p) = \tfrac12 \d\log(|s_0|^2) = \tfrac12
|s_0|^{-2}\, \d(|s_0|^2)
 = |s_0|^{-2}\, \Re \langle \nabla s_0, s_0\rangle
= 0. $$
Moreover, 
$$2\pi \omega_p + (\d\d^cf_0)_p=\d\d^c\log \left|
\frac{s}{s_0} \right|=-i\d'\d''\log\left|
\frac{s}{s_0} \right|^2=-(i\d'\d'' \log
|F|^2)_p=0$$
because the $1$-jet of the form $\d''F$ vanishes
at $p$. Therefore, there exists a constant $C>0$
such that $|2\pi \omega + \d\d^cf_0|_g\le C d_1$.
Multiplicating this by $k$ gives the first bound
of the statement.
On the other hand, the Hessian quadratic form
$(\d^2f_0)_p: T_pX \to \R$ vanishes on $T_pQ$
and satisfies, for every vector $v \in T_pX$,
$$ (\d^2f_0) (v,v) + (\d^2f_0) (iv,iv) =
-(\d\d^cf_0) (v,iv)
 = 2\pi\omega(v,iv) = 2\pi g(v,v). $$
Hence, $(\d^2f_0)_p$ is positive definite on
$i\,T_pQ$ and $Q$ is a Morse-Bott minimum for
$f_0$. Since $Q$ is compact, one can find a
constant $C>0$ such
that, on some neighbourhood of $Q$ for the metric
$g_1$:
$$ d_1^2/C \le f_0 \le  Cd_1^2. $$
In other words, $e^{-Cd_1^2}\le |s|\le
e^{-d_1^2/C}$. We obtain the second bound of the
statement by taking the $k$-th power. The third
bound and the bounds
\begin{align*}
 |\nabla''s_0^k|_{g} &\le C  k d_1^{r+1}
e^{-kd_1^2/C}, &
 |\nabla\nabla''s_0^k|_{g} &\le C k
d_1^{r}(1+kd_1^2) e^{-kd_1^2/C}
\end{align*}
follow from this bound and the bounds on $\nabla
s_0$, $\nabla''s_0$ and
$\nabla\nabla''s_0$ by the Leibniz rule applied to
$s_0^k$. The two latter real-valued Gaussian
functions of $d_1$ both reach their global maximum
at $Constant \times k^{-1/2}$. By expressing these
bounds in the rescaled metric $g_k$, we obtain the
last bound of the statement.
\end{proof}

The following is the main result of this section.

\begin{prop}\label{d''}
Let $\rho\in (0,c)$. There exist holomorphic sections $s_k : X \to L^k$
such that, for every $\ep>0$ and for $k\ge
k_0(\ep)$ sufficiently large, $s_k$ 
vanishes transversally and
$\|s_k-s_0^k\|_{\c^1,g_k} < \ep$ on $B_1(Q,\rho)$, the $\rho$-neighbourhood of $Q$ in the metric $g$.
\end{prop}

We postpone the proof of proposition \ref{d''} and
first explain how it implies
theorem \ref{section hyperplane complexe}.

\begin{proof}[Proof of theorem \ref{section
hyperplane complexe}]

We fix a radius $\rho\in (0,c)$ and, by proposition \ref{d''}, holomorphic sections $s_k: X \to L^k$: for every $\ep>0$ and for $k\geq k_1(\ep)$
sufficiently large, the zero-set $Y :=
s_k^{-1}(0)$ is a (smooth) complex
hyperplane section and $\Vert s_k-s_0^k\Vert
_{\c^1,g_k} < \ep$ on $B_1(Q,\rho)$. By
the second and third inequalities in lemma
\ref{decroissance:exponentielle}, there
exists a constant $C>0$ (independent of $k$ and
$\ep$) such that, for $\ep>0$ sufficiently small,
on $B_k(Q,\rho)$, the functions $f_1 :=
-\log|s_k|$ and $f_0 :=
-\log|s_0^k|$ satisfy
\begin{align*}
 \Vert f_1-f_0\Vert _{\c^1,g_k} < C\ep.
\end{align*}

Take a cutoff function $\beta_k: X \to [0,1]$
supported in $B_k(Q,\rho)$, with $\beta_k=1$ on
$B_k(Q,\rho/2)$ and $\Vert \beta_k \Vert
_{\c^2,g_k} \le C'$ for some
constant $C'>0$ (independent of $k$). The function
$f := \beta_k f_0
+  (1-\beta_k)f_1: X \setminus Y \to \R$ is
exhausting, reaches a Morse-Bott minimum at $Q$
and its critical points remain in a compact
subset. (We remark that, for sufficiently small
$\ep$, this minimum is global. Indeed, on
$\{\beta_k=1\}$, $f=f_0$, and on $\{\beta_k<1\}$,
$f_1\ge -\log(|s_0|+\ep)\ge
-\log(e^{-\rho^2/C}+\ep)>0$.)

Let us show that $f$ is $\C$-convex. First, since
$s_k$ is holomorphic, $-\d\d^c f_1=2k\pi \omega$.
Then, by the first bound of lemma
\ref{decroissance:exponentielle}, there exists a
constant $C''>0$ such that
$\|\d\d^c(f_0-f_1)\|_{\c^0,g_k}\le C''k^{-1/2}$.
Hence,
\begin{align}
& \Vert 2k\pi \omega+ \d\d^c f \Vert_{\c^0,g_k}
\nonumber \\
& = \Vert \d\d^c(\beta_k(f_0-f_1))
\Vert_{\c^0,g_k} \nonumber \\
& \le \Vert \beta_k \Vert C''k^{-1/2} + \Vert
(f_0-f_1)\d\d^c \beta_k \Vert + \Vert  \d
(f_0-f_1) \wedge \d^c \beta_k  \Vert + \Vert   
\d^c (f_1-f_0)\wedge\d \beta_k   \Vert \nonumber
\\
& \le   C''k^{-1/2} + 3(C\ep) C'.
\label{C-convexite}
\end{align}
Consequently, for every $\ep >0$ sufficiently
small and for every $k \ge k_0(\ep)$ sufficiently
large, $\Vert 2k\pi \omega+ \d\d^c f
\Vert_{\c^0,g_k}  <2\pi$. This inequality ensures
that the
function $f$ is $\C$-convex. A $\c^2$-small
perturbation of the function $f$ with support in a
compact subset of $Y\setminus Q$ is Morse away
from $Q$ and satisfies the properties of theorem
\ref{section hyperplane complexe}.
\end{proof}

Our next aim is to prove proposition \ref{d''}. In
the following lemma the $\mathrm L^2$-norm of a
section $s:X\to \bigotimes^rT^\star X\otimes L^k$
for the rescaled metric $g_k$ is defined by
$$\Vert s \Vert_{\mathrm L^2,g_k}:=\left(\int_X
|s|^2_{g_k} \tfrac{(k\omega)^n}{n!}
\right)^{1/2}.$$
\begin{lem} \label{conversion}
Let $\beta: X \to [0,1]$ a function
supported in $N$
with $\beta=1$ on a tube $B(Q,\rho)$. There exists a constant $C>0$ such that the sections $s_{0,k} :=
\beta s_0^k:X\to L^k$ satisfy the following bounds:
$$\|\nabla''s_{0,k}\|_{\c^1,g_k}\le Ck^{-r/2}, 
\|\nabla''s_{0,k}\|_{\mathrm L^2,g_k}\le
Ck^{(n-r)/2}$$
\end{lem}

\begin{proof}
The sections $s_0^k$ satisfy the bounds of
Lemma \ref{decroissance:exponentielle} on
$N$. Then, there exists a constant $C>0$ such that:
\begin{align*}
   \|\nabla''s_{0,k}\|_{\c^0,g} &\le
   \|\d\beta \|_{\c^0,g} \sup_{\{d_1 > \rho\}}
|s_0^k| 
 + \sup_{B(Q,2\rho)} |\nabla''s_0^k|_{g } \\
& \le C (e^{-k/C} +  k^{-(r-1)/2}) \ .
\end{align*}
In the same way:
\begin{align*}
&\|\nabla\nabla''s_{0,k}\|_{\c^0,g } \\
&\le  \|\d^2\beta \|_{\c^0,g} \sup_{\{d_1 >
\rho\}} |s_0^k|
 + 2 \|\d\beta \|_{\c^0,g } \sup_{\{d_1 > \rho\}}
|\nabla s_0^k|_{g}
 + \sup_{B(Q,2\rho)} |\nabla\nabla''s_0^k|_{g } \\
&\le  C e^{-k/C} + C e^{-k/C} + Ck^{-(r-2)/2} .
\end{align*}
Since
\begin{align*}
 \|\nabla''s_{0,k}\Vert _{\mathrm L^2,g_k } \le
    C k^{n/2}\|\nabla''s_{0,k}\|_{\c^0,g_k},
\end{align*}
the $\c^1$ and the $\mathrm L^2$ norms, in the
metric $g_k$, satisfy the bounds of the statement.
\end{proof}

We now use the following version of Hörmander's
$\mathrm L^2$-estimates:

\begin{theo}[{cf. \cite[theorem VIII.6.5]{Dem12}
and the discussion thereafter}]\label{Hor:65}
Let $(X,\omega)$ be a closed integral Kähler
manifold and $L \to X$ a
holomorphic Hermitian line bundle with Chern
curvature $-2\pi i \omega$. Set $C:=\sup |
\tfrac{\mathrm{Ricci}(\omega)}{2\pi}|_g$. Then,
for every $k > C$ and for every smooth section
$u:X\to \bigwedge^{1,0}T^\star X\otimes L^k$ such
that $\nabla''u = 0$, there exists a smooth
section $t:X\to L^k$ satisfying:
\begin{align*}
\nabla''t = u \text{ and }\ \Vert t\Vert^2
_{\mathrm L^2}\le  \frac{1}{n(k-C)} \Vert  u
\Vert^2 _{\mathrm L^2} \ .
\end{align*}
\end{theo}

Applying this theorem to the sections $s_{0,k}$ of
lemma \ref{conversion}, we
obtain smooth sections $t_k:X\to L^k$ satisfying
$\Vert t_k \Vert_{\mathrm L^2,g_k}\leq
Ck^{(n-r-1)/2}$, and, for $k$ sufficiently large,
$\nabla''(s_{0,k} - t_k)=0$. The following lemma
converts our $\mathrm L^2$-estimates to
$\c^1$-estimates.
\begin{lem}\label{L^2:C^1}
Let $(X,\omega)$ be a closed integral Kähler
manifold, $L\to X$ a holomorphic Hermitian line
bundle with Chern curvature $-2 \pi i\omega$.
There exists a constant $C>0$ such that for every
integer $k$ and for every section $t:X\to L^k$:
\begin{align*}
\Vert t\Vert _{\c^1,g_k}\le C (\Vert \nabla''
t\Vert _{\c^1,g_k} + \Vert t\Vert _{\mathrm
L^2,g_k}).
\end{align*}

\end{lem}

\begin{proof}
The desired bound is local. At a given point $p\in
X$, we will obtain it on a $g_k$-ball of uniform
radius about $p$ --- where, for sufficenlty large
$k$, the geometry of $L^k$ compares with the
trivial line bundle over the unit ball of
euclidean space $(\C^n,g_0)$. There exist
constants $r,C>0$ and a family (indexed by $p,k$)
of holomorphic charts $\underline z_p^k:
B_k(p,r)\to \C^n$ centered at $p$ such that,
\begin{align}
\Vert (\underline
z_p^k)_{\star}g_k-g_0\Vert_{\c^1,g_0} \le C
k^{-1/2} \text{ over $(\underline
z_p^k)(B_k(p,r))$.} \label{holomorphic:chart}
\end{align}
Indeed, we take some constant $r>0$ and a smooth
family (indexed by $p\in X$) of centered
holomorphic charts $\underline z_p: B(p,r)\to
\C^n$ with isometric differentials $\d \underline
z_p (p)$ at the origin. The post-composition of
$\underline z_p$ by the centered dilation
$\C^n\to\C^n$ of ratio $k^{1/2}$ gives $\underline
z^k_p$.
 
Let us take a Hörmander holomorphic peak section
at $p$ (see for instance \cite[Proposition
34]{Don96}): for sufficiently large $k$, there
exists a holomorphic section $s_{p}:X\to L^k$
satisfying the bounds:
\begin{align*}
|s_p(p)|=1 , \inf_{B_k(p,r)} |s_p|\ge C^{-1}
\text{ and }\Vert s_p\Vert _{\c^1,g_k}\le C,
\end{align*}
for some constant $C>0$ independent of $p$ and
$k$.

Let $t$ be a section of $L^k$ and $p\in
B(Q,\rho')$. We set $f := \frac{t}{s_p}$. In view
of the identities $\nabla t = \d f \ s_p + f
\nabla s_p$, $\nabla\nabla t = \d^2 f \ s_p +2\d
f\otimes \nabla s_p + f \nabla \nabla s_p$, and
the bounds on the peak sections, it suffices to
show that for sufficently large $k$, $$\Vert
f\Vert _{\c^1(B_k(p,r/6)),g_k}\le C\Vert \d''
f\Vert _{\c^1(B_k(p,r)),g_k}+C \Vert f\Vert
_{\mathrm L^2(B_k(p,r)),g_k}.$$ In the following,
we will identify the domain of the chart
$\underline z^k_p$ with its image in $\C^n$. We
denote by $B_0(q,r)$ the ball of radius $r$ at a
point $q$ in $\C^n$ and by $\mu$ the Euclidean
volume form on $\C^n$. Let us prove the (standard)
following bound:
$$\Vert f\Vert _{\c^1(B_0(0,r/5)),g_0}\le C\Vert
\d'' f\Vert _{\c^1(B_0(0,r/2)),g_0}+ C \Vert
f\Vert _{\mathrm L^2(B_0(0,r/2)),g_0}. $$
This will end the proof because, in view of the
comparaison (\ref{holomorphic:chart}) of the
rescaled metric $g_k$ with the flat metric $g_0$,
for sufficiently large $k$, we have the inclusions
$B_k(p,r/6) \subset B_0(0,r/5)$ and
$B_0(0,r/2)\subset B_k(p,r)$, and there exists a
constant $C>0$ (independent on $k$ and $p$) such
that, over $B_0(0,r/2)$,
\begin{align*}
\mu  \le (1+C k^{-n/2}) \tfrac{(k\omega)^n}{n!}
\text{ and }(1-Ck^{-1/2}) |\cdot|_{g_0} \le 
|\cdot|_{g_k}\le (1+Ck^{-1/2}) |\cdot|_{g_0} \ .
\end{align*}

On the one hand, \cite[Lemma 4.4]{HW68} gives:
\begin{align*}
 \Vert f\Vert _{\c^0(B_0(0,r/4))} \le C \Vert \d''
f\Vert _{\c^0(B_0(0,r/2))} + C\Vert f\Vert
_{\mathrm L^2(B_0(0,r/2)),g_0} \ .
\end{align*}
On the other hand, we have the following standard
bound (cf. \cite[Lemma 8.37]{CE12} for instance):
\begin{align*}
 \Vert f\Vert _{\c^1(B_0(0,r/5)),g_0} \le C \Vert
\d'' f\Vert _{\c^1(B_0(0,r/4)),g_0} + C\Vert
f\Vert _{\c^0(B_0(0,r/4))} \ .
\end{align*}
In the two above estimates the constants depend
only on $r$ and $n$. Therefore we obtain the
desired bound.
\end{proof}

By lemma \ref{L^2:C^1}, we obtain the following
estimate: for every $\ep>0$, for $k\ge k_1(\ep)$
sufficiently large,
$$\Vert t_k \Vert_{\c^1,g_k}\le
C(\| \nabla''
s_{0,k}\|_{\c^1,g_k}+k^{-1/2}\|s_{0,k}\|_{\mathrm
L^2,g_k}) \le Ck^{(n-r-1)/2}<\ep/2.$$
On the other hand, by Bertini theorem, for
sufficiently large $k$ there exists a holomorphic
section $s_k:X\to L^k$ vanishing transversally
with $$\Vert s_k - (s_{0,k}-t_k)\Vert
_{\c^1,g_k}<\ep/2.$$ Therefore the sections $s_k$
satisfy the conclusions of proposition \ref{d''}.
This ends the proof of theorem \ref{section
hyperplane complexe}.
 
Let us finish with a complex-geometric variant of
theorem \ref{section hyperplane complexe}:
\begin{theo}\label{postrem}
Let $X$ be a closed complex manifold, $a$ a Kähler class
and $Q$ a closed submanifold. Suppose that $Q$ is a
Bohr-Sommerfeld Lagrangian submanifold for some Kähler form in
$a$. Then, there exists a holomorphic line budle $L\to X$ with first Chern class $a$,
 and, for every sufficiently large
$k$, there exist a Hermitian metric $h_k$ on $L^k$ with positive
Chern curvature and a holomorphic section
$s_k:X\to L^k$ vanishing transversally such that
the function $-\log|s_k|_{h_k}: X \setminus s_k^{-1}(0)
\to \R$ has a Morse-Bott minimum at $Q$ and is
Morse elsewhere.
\end{theo}

\begin{proof}[Proof of theorem \ref{postrem}]
We fix a Kähler form $\omega\in a$ with
$\omega|_Q=0$ as well as a Hermitian holomorphic line bundle
$L\to X$ with Chern curvature $-2i\pi \omega$ whose restriction to $Q$ is a trivial flat bundle (by lemma \ref{fibré}). We fix $\ep,\rho>0$ and repeat the
construction of section \ref{K} to obtain sections $s_0^k, s_k :
X\to L^k$ with the properties stated in lemma
\ref{decroissance:exponentielle} and proposition
\ref{d''}. We keep the notation $f_0=-\log
|s_{0,k}|$ and $f_1=-\log |s_k|$.

To construct the desired Hermitian metric on
$L^k$, we will proceed as in the final step of the
proof of theorem \ref{section hyperplane complexe}
but we will modify the initial Hermitian metric $h^k$ of
$L^k$ instead of the function $f_1$. Take a cutoff
function $\beta_k  : X\to [0,1]$ with support in
$B_k(Q,\rho)$ with $\beta_k =1$ on $B_k(Q,\rho/2)$
and such that $\Vert \d \beta_k \Vert _{\c^1,g_k}
<C'$, for some constant $C'>0$ independent of $k$.
We define a new  Hermitian metric on $L^k$
by:
$$h'_k=e^{2\beta_k  (f_1-f_0)} h^k.$$
The exhaustion function $-\log
|s_k|_{h'_k}:\{s_k\neq 0\}\to \R$ equals $f_0$ on $B_k(Q,\rho/2)$ hence has a Morse-Bott
local minimum at $Q$.
Furthermore, $$ 2k\pi\omega - \d\d^c \log
|s_k|_{h'_k} = \d\d^c (\beta_k  (f_1-f_0)).  $$
Therefore, by repeating the estimation
(\ref{C-convexite}), for every $\ep <\ep_0$
sufficiently small and for $k \ge k_0(\ep)$
sufficiently large, $\Vert 2k\pi\omega - \d\d^c
\log |s_k|_{h'_k} \Vert_{\c^0,g_k}  <2\pi$. This
inequality ensures that the
function $-\log |s_k|_{h'_k}$ is $\C$-convex.
Finally, there exists a $\c^2$-small function
$\eta_k: X\setminus Y\to \R$ with compact support
away from $Q$ such that, setting the Hermitian
metric $h''_k:= e^{-2\eta_k}h'_k$, the function $-\log
|s_k|_{h''_k} = -\log |s_k|_{h'} + \eta$ is Morse
away from $Q$.

In conclusion, the Hermitian metric $h''_k$ and the
sections $s_k:X\to L^k$ have the desired
properties.

\end{proof}


\begin{thebibliography}{AGM01}

\frenchspacing

\bibitem[AGM01]{AGM01}
D. Auroux, D. Gayet, and J.-P. Mohsen. \ 
Symplectic hypersurfaces in the complement of an
isotropic submanifold. \ 
\textit{Math. Ann.} \textbf{321} (2001), 739--754.

\bibitem[AMP05]{AMP05}
D. Auroux, V. Muñoz, and F. Presas. \ 
Lagrangian submanifolds and Lefschetz pencils. \ 
\textit{J. Sympl. Geom.} \textbf{3} (2005),
171--219.

\bibitem[CM17]{CM17}
K. Cieliebak, K. Mohnke. \
Punctured holomorphic curves and Lagrangian
embeddings. \
\textit{Invent. math.} (2017)
(https://doi.org/10.1007/ s00222-017-0767-8).

\bibitem[CE12]{CE12}
K. Cieliebak and Y. Eliashberg. \ 
\textit{From Stein to Weinstein and Back --
Symplectic Geometry of Complex
Affine Manifolds}. \
Colloq. Publ. \textbf{59}, Amer. Math. Soc., 2012.

\bibitem[Dem12]{Dem12}
J.-P. Demailly. \ 
\textit{Complex Analytic and Differential
Geometry}.
Preprint 2012, Inst. Fourier, Grenoble
(https://www-fourier.ujf-grenoble.fr/~demailly/man
uscripts/agbook.pdf).

\bibitem[Don96]{Don96}
S. K. Donaldson. \ 
Symplectic submanifolds and almost-complex
geometry. \
\textit{J. Diff. Geom.} \textbf{44} (1996),
666--705.



\bibitem[DS95]{DS95}
J. Duval and N. Sibony. \
Polynomial convexity,rational convexity, and
currents.
\textit{Duke Math. J.} \textbf{79} (1995)
,487--513.

\bibitem[EGL15]{EGL15}
Y. Eliashberg, S. Ganatra, and O. Lazarev. \ 
\textit{Flexible Lagrangians}. \ 
Preprint 2015 (arXiv.org/abs/1510.01287).

\bibitem[Gir18]{Gir18}
E. Giroux. \ 
\textit{Remarks on Donaldson's symplectic
submanifolds}. \ 
Preprint 2018 (arxiv.org/abs/1803.05929).

\bibitem[Gra58]{Gra58}
H. Grauert. \  
Analytische Faserungen über
holomorph-vollständigen Räumen. \ 
\textit{Math. Ann.} \textbf{135} (1958),
263--273. 

\bibitem[Gue99]{Gue99}
V. Guedj. \  
Approximation of currents on complex manifolds. \ 
\textit{Math. Ann.} \textbf{313} (1999), 437--474.

\bibitem[Oka39]{Oka39}
K. Oka. \ 
Sur les fonctions des plusieurs variables. III:
Deuxième problème de Cousin. \ 
\textit{J. Sc. Hiroshima Univ.} \textbf{9} (1939),
7--19.

\bibitem[HW68]{HW68}
L. Hörmander and J. Wermer. \ 
Uniform approximation on compact sets in $\C^n$.
\ 
\textit{Math. Scand.} \textbf{223} (1968), 5--21.


\bibitem[Thu97]{Thu97}
W. P. Thurston. \ 
\textit{Three--Dimensional Geometry and Topology}.
\  
Princeton Math. Ser. \textbf{35}, S. Levy ed.,
Princeton Univ. Press, 1997.



\bibitem[Wei95]{Wei95}
C. A. Weibel. \
\textit{An Introduction to Homological Algebra}.
\ 
Cambridge Univ. Press, 1995.

\end{thebibliography}
\end{document}